\documentclass[a4paper,12pt]{article}
\usepackage{amsmath,amsfonts,amssymb,amsthm,amscd}
\usepackage{pxfonts}
\usepackage{mathrsfs}
\usepackage{color}
\usepackage[all]{xy}
\usepackage{stmaryrd}
\newcommand{\field}[1]{\mathbb{#1}}

\newcommand{\Q}{\field{Q}}

\newtheorem{theorem}{Theorem}[section]
\newtheorem{lemma}[theorem]{Lemma}
\newtheorem{proposition}[theorem]{Proposition}
\newtheorem{corollary}[theorem]{Corollary}
\theoremstyle{definition}
\newtheorem{example}[theorem]{Example}
\newtheorem{definition}[theorem]{Definition}

\theoremstyle{remark}
\newtheorem{remark}[theorem]{Remark}

\DeclareMathOperator{\gal}{Gal}
\DeclareMathOperator{\irr}{Irr}
\allowdisplaybreaks
\numberwithin{equation}{section}
\begin{document}
\title{On multiply induced characters}
\author{Masanari Kida\thanks{This work is supported by JSPS KAKENHI Grant Number 25K06962.}, Kyohei Matsukawa,
 and Hikari Watanabe}
\date{\today}
\maketitle
\renewcommand{\thefootnote}{}
\renewcommand{\thefootnote}{\arabic{footnote}}
\setcounter{footnote}{0}

\begin{abstract}
    We investigate finite group structures that yield complex irreducible characters induced simultaneously 
    from two or more distinct subgroups. When such a group is realized as the Galois group of a Galois extension
    of number fields, this phenomenon implies that
    multiple subfields share the same Artin $L$-function.
\end{abstract}

\section{Introduction} \label{sec:1}
Let $G$ be a finite group. We say that a complex character $\chi$ of $G$ is \emph{multiply induced}
if there are proper subgroups $H_1$ and $H_2$ of $G$ such that $\chi$ is induced from a character $\psi_1$ 
of $H_1$ and simultaneously from $\psi_2$ of $H_2$.

As detailed in the next section, if $G$ is realized as a Galois group of a number field extension  
and admits a multiply induced character, then the induction property of the Artin $L$-function implies 
a coincidence of $L$-functions. To be more specific,
the $L$-functions associated with the characters $\psi_1$ and $\psi_2$ over the subfields corresponding to $H_1$ and $H_2$  
are identical. Consequently, these subfields are 
expected to  share significant arithmetic properties. Thus, the existence of multiply induced characters
is a phenomenon of interest not only in group representation theory but also in algebraic number theory.

A classical and well-known example of multiply induced characters comes 
from a Gassmann triple $(G,H_1,H_2)$, which implies a character identity
$(1_{H_1})^G = (1_{H_2})^G$. In the context of number theory, 
such a triple results in two
number fields sharing the Dedekind zeta function, and these number fields are called arithmetically 
equivalent (see Section \ref{sec:2}). 
Multiply induced characters provide a natural 
generalization of this classical equivalence.

While many examples of groups with multiply induced characters have been found, 
to the best of our knowledge, 
this is the first systematic study of irreducible characters 
that are multiply induced from distinct proper subgroups.
In particular, the purpose of this paper is 
to establish a necessary and sufficient condition for the existence of 
multiply induced characters from maximal normal subgroups 
under the assumption of the existence of a Camina triple,
whose definition is as follows (\cite{MR3150725}):

\begin{definition} \label{def:1.1}
    Let $G$ be a finite group.
A triple $(G, N, M)$ consisting of nontrivial proper normal subgroups $N$ and $M$ with $N \ge M$ of $G$ is 
called a \emph{Camina triple} if every $g \in G \backslash N$ is conjugate to all elements of $gM$.
\end{definition}

If an induced character is not irreducible,
the corresponding Artin $L$-function factors into a product of $L$-functions attached
 to its irreducible constituents. 
 Consequently, our primary interest lies in whether an irreducible character 
 can be multiply induced. 
 To guarantee the irreducibility of such induced characters,
  the structural rigidity of a Camina triple plays an essential role
   (as demonstrated in Lemma 3.2).

Our main result is the following.
\begin{theorem} \label{thm:1.1}
    Let $G$ be a finite solvable group and $N_1$ and $N_2$ non-trivial maximal normal subgroups 
    of $G$. Let $H=N_1 \cap N_2$.
    Assume that $(G,N_1, M_1)$ is a Camina triple. If
    there exist $\psi \in \irr (N_1 \,|\, M_1) $ and $\alpha \in \irr ( H  \,|\, \psi)$
    satisfying one of the following conditions:
    \begin{enumerate}
        \item \label{1.1.1} if $\gcd (|G/N_1|, |G/N_2 |)=1$, then we assume that $I_G (\alpha ) =H$;
        \item \label{1.1.2} if $\gcd (|G/N_1|, |G/N_2 |) \ne 1$, then we assume that $I_G (\alpha )$ coincides with neither 
        $N_1$ nor $N_2$,
    \end{enumerate}
 then there exists $\gamma \in \irr (N_2)$ such that $\gamma^G = \psi^G$.

 Conversely, if there exists $\gamma \in \irr (N_2)$ satisfying $\gamma^G = \psi^G$ for a given $\psi \in \irr (N_1 \, | \, M_1)$, then
  $\alpha \in \irr (H \, | \, \psi) \cap \irr (H \, | \, \gamma ) $ must satisfy one of the above conditions \ref{1.1.1} and \ref{1.1.2}.
\end{theorem}

For the notation used in the statement, see the end of this section.
The solvability assumption is essential for our approach, 
as the proof relies on structural properties of maximal normal subgroups.
Thus, under the assumption of the existence of a Camina triple, 
the conditions in Theorem 1.2 are necessary and sufficient, and therefore
Theorem \ref{thm:1.1} works well for inductions from normal subgroups.
On the other hand, there are groups admitting multiply induced characters
from non-normal subgroups and having no Camina triples (see Example \ref{ex:5.7}).

The theorem will be proved in Section \ref{sec:3} after
we give number-theoretic significance and consequence of the theorem in Section \ref{sec:2}.
Although our motivation partly comes from number-theoretic considerations as we will see, 
the main results of this paper are purely character-theoretic and concern 
the structure of induced irreducible characters in finite solvable groups.

In Section \ref{sec:4}, we show the following theorem.

\begin{theorem} \label{thm:1.2}
    If $G$ and $H$ are isoclinic finite groups and there exists a multiply induced irreducible character in $G$,
    then there exists such a character also in $H$.
\end{theorem}

We also prove 
that there exists a one-to-one correspondence between Camina triples of isoclinic groups $G$ and $H$ 
(see Theorem \ref{prop:4.5}).

In Section \ref{sec:5}, we give explicit examples of Theorem \ref{thm:1.1}. 

Throughout this paper, all groups considered are finite groups
and the following notation will be used.
For a finite group $G$, we denote by $Z(G)$ the center of $G$ and by
$G'$ the commutator subgroup of $G$.
If a group is designated by the double $(n,k)$, then it always means a GAP ID in 
the SmallGroups library, namely the group is the $k$-th group of order $n$.

All representations appearing in this paper are linear complex representations.
We denote by $\irr (G)$ the set of the irreducible characters of $G$ and
$\irr (G)_k$ the subset of degree-$k$ characters in $\irr (G)$.
Let $H$ be a subgroup of $G$. For $\chi \in \irr  (G)$, we denote by 
$\chi_H$ the restriction of $\chi$ to $H$ and also for $\xi \in \irr (H)$,
by $\xi^G$ the induction of $\xi$ to $G$. 
We define
\begin{align*}
& \irr (G \,|\, \xi) = \{ \chi \in \irr (G) \mid (\chi , \xi^G) >0 \}, \\
& \irr (H \,|\, \chi ) = \{ \xi \in \irr (H) \mid (\chi_H, \xi ) >0 \}.
\end{align*}
For a normal subgroup $N$ of $G$, we also define
\[
\irr (G \,|\, N ) = \irr (G) - \irr (G/N),
\]
which can be identified as the set of the irreducible characters of $G$ whose kernel does not contain $N$. 
If $G$ acts on a character $\psi$, we denote by $I_G (\psi ) = \{ g \in G \mid \psi^g = \psi \}$ the 
stabilizer (the inertia group) of $\psi $.

Finally, a $G$-extension is a Galois extension over the field $\Q$ of rationals
whose Galois group is isomorphic to $G$.

\section{Number-theoretic background and consequences} \label{sec:2}
In this section, we illustrate number-theoretic background that motivates our study on multiply induced characters.
We first recall the definition of Artin $L$-function. 

\begin{definition}[\protect{\cite{MR3069563}}]
Let $K/k$ be a Galois extension of number fields with Galois group $G$.
Let $\rho $ be a finite-dimensional linear representation of $G$ on a complex vector space $V$.
The Artin $L$-function is defined by the Euler product
\[
L( s, \rho , K/k) = \prod_{\mathfrak{p}}  \det 
\left( \left. 1 - \frac{\rho(\mathrm{Frob}_{\mathfrak{P}})}{N \mathfrak{p}^{s}} \right| 
V^{I_{\mathfrak{P}}}\right)^{-1},
\]
where $\mathfrak{p}$ runs through all the prime ideals of $k$, $\mathrm{Frob}_{\mathfrak{P}}$ is 
the Frobenius automorphism at a fixed prime $\mathfrak{P}$ of $K$ lying above $\mathfrak{p}$, and $I_{\mathfrak{P}}$ is 
the inertia group of $\mathfrak{P}$.
\end{definition}

The infinite product defining the Artin $L$-function converges absolutely and uniformly in $\mathrm{Re} (s) >1$.
The $L$-function has a meromorphic continuation to the whole complex plane and 
it is conjectured to be entire if $\rho $ is non-trivial and irreducible.

Roughly speaking, the Artin $L$-function encodes the behavior of a prime decomposition 
in $K/k$, that is, which conjugacy class the Frobenius automorphism lies in.

If $K/k$ is an abelian extension, the above Artin $L$-function coincides with the Hecke $L$-function
associated with certain ray class character. If $\rho$ is a regular representation, 
the Artin $L$-function coincides with the Dedekind zeta function.

Let $K_1$ and $K_2$ be number fields of finite degree over the rational field $\Q$.
If the Dedekind zeta functions of $K_1$ and $K_2$ coincide, then the fields $K_1$ and $K_2$ are said to be
\emph{arithmetically equivalent} (over $\Q$). 
It is known that arithmetically equivalent fields share the same Galois closure over $\Q$, and thus, 
if $K_1$ and $K_2$ are arithmetically equivalent, then they lie 
in a common Galois closure $L$ over $\Q$.
Let $G=\gal (L/\Q)$ and $H_1$ and $H_2$ be subgroups of $G$ corresponding to $K_1$ and $K_2$.
In his paper \cite{zbMATH02585915}, Gassmann reduced this arithmetic 
problem of finding arithmetically equivalent fields
to a purely group theoretic setting. In fact, he proved that 
$K_1$ and $K_2$ are arithmetically equivalent if and only if the groups $H_1$ and $H_2$ 
are Gassmann equivalent in $G$; that is, the permutation characters satisfy $(1_{H_1})^G = (1_{H_2})^G$.
In our terminology, the character is multiply induced from $H_1$ and $H_2$.
Gassmann equivalence follows from the fact that the Dedekind zeta function of $K_1$ is 
nothing but the Artin $L$-function associated with $(1_{H_1})^G$ 
This is the first occurrence of multiply induced characters in number theory.

In 1925, Hecke found that, in a $D_4$-extension (of degree $8$), the unique degree-$2$ character of the Galois group
is induced from the three subgroups corresponding to the quadratic subfields
and used to define
theta series associated with indefinite quadratic forms \cite{zbMATH02591849}.
As is noticed above, Hecke's discovery can be stated in terms of Artin $L$-functions
by virtue of the following induction property of 
Artin $L$-function.
 
\begin{proposition}[\protect{\cite[Satz 1]{MR3069563}}] \label{prop:2.1}
Let $L/K$ be a Galois extension with Galois group $G$ and $M$ an intermediate field of $L/K$ with 
$H=\gal (L/M)$. If $\psi $ is a character of $H$, then 
\[
L (s, \psi^G , L/K ) = L (s, \psi , L/M).
\]
\end{proposition}

Hecke's result is, thus, equivalent to that there are linear characters $\psi_i$ of $\gal (L/M_i)\ (i=1,2,3)$ such that 
\[
L (s, \psi_1 , L/M_1 ) = L (s, \psi_2, L /M_2) = L (s , \psi_3 , L/M_3) 
\]
holds, where $L/\Q $ is a $D_4$-extension, and $M_i$ are the quadratic subfields of $L$.
Note that the above equality holds only up to finite number of Euler factors at ramified primes
(see \cite[VII.10.6]{MR1697859}).
Henceforth, we understand that the equality of abelian $L$-functions is up to finite number of Euler factors.

Since Hecke's paper, such ``coincidence'' of abelian $L$-functions has been studied by several authors:
\cite{MR914305}, \cite{MR3581913}, and \cite{MR4474773}.
The following result supersedes these results.
 
\begin{proposition}[\protect{\cite{10.55937/sut/1685448649}}] \label{prop:2.2}
Let $p$ be a prime and $n$ a positive integer. If $L/\Q$ is a Galois
extension with Galois group isomorphic to an extraspecial group $\mathrm{ES}(n,p)$ of order
$p^{2n+1}$, then there exist $N$ elementary abelian subextensions $M_i$ of $L$ 
of degree $p^n$ and linear characters $\chi_i$ of $\gal (L/M_i)$
such that $L (s,  \chi_i, L/M_i)$ are identical.
Here $N$ is equal to $\prod_{i=1}^n (p^i +1)$.
\end{proposition}

This proposition, of course, paraphrases the fact that an irreducible character of $\mathrm{ES}(n,p)$ of degree $p^n$ 
is multiply induced from $N$ linear characters of abelian subgroups of index $p^n$.
While this fact is widely considered folklore for experts, most number theorists regarded it as 
a rare or sporadic phenomenon. 
Our Theorem \ref{thm:1.1} shows that it is not only a coincidence, but a structural inevitability from 
the subgroup configuration of Galois group.
This paper also aims to bridge the gap between these perceptions.

\begin{theorem} \label{thm:2.3}
    Let $L/\Q $ be a Galois extension with Galois group $G$.
    Assume that $G$ satisfies the conditions in Theorem \ref{thm:1.1}. In the notation of Theorem \ref{thm:1.1},
    let  $K_1$ and $K_2$ be the subfields corresponding to $N_1$ and $N_2$, respectively.
    Then there exist $\chi \in \irr (G), \, \psi_i \in \irr (H_i) \, (i=1,2)$ satisfying $\psi_i^G = \chi$,
    and hence we have the equality:
    \begin{equation} \label{eq:2.1}
    L (s, \chi, L/\Q ) =L (s, \psi_1, L/K_1) = L (s, \psi_2, L/K_2 ).
    \end{equation}
\end{theorem}

It is readily seen that Theorem \ref{thm:2.3} follows from Theorem \ref{thm:1.1} and Proposition \ref{prop:2.1}.
We will see in Section \ref{sec:5.2} that Proposition \ref{prop:2.2} follows from Theorem \ref{thm:2.3}, but
it is also worth noting that Theorem \ref{thm:2.3} says much more than Proposition \ref{prop:2.2}.
In particular, $\psi_i$ is not necessarily a linear character. Indeed, we can provide
examples of primitive characters $\psi_i$ of arbitrarily large degree satisfying \eqref{eq:2.1}
 (see Section \ref{sec:5.3}).
It is noteworthy that, in the notation of Theorem \ref{thm:2.3},
the equality $L(s,\psi_1,L/K_1)=L(s,\psi_2,L/K_2)$ implies
$\psi_1^G=\psi_2^G$ by \cite[Proposition 3]{KlunersNicolae2016}.

Combining the above theorem with Theorem \ref{thm:1.2}, we obtain the following theorem with weaker assumption.

\begin{theorem} \label{thm:2.4}
    Let $L/\Q $ be a Galois extension with Galois group $H $.
    Assume that $H$ is isoclinic to a group $G$ satisfying the conditions in Theorem \ref{thm:1.1}. 
    Then there exist subfields $K_1$ and $K_2$  of $L$ and characters $\psi_1$ and $\psi_2$ of 
    $\gal (L/K_1) $  and $\gal (L/K_2) $ satisfying 
    \[ L (s, \psi_1, L/K_1) = L (s, \psi_2, L/K_2 ). \]
\end{theorem}

The subfields $K_1$ and $K_2$ and the characters $\psi_1$ and $\psi_2$ are explicitly given in terms of 
the isoclinism between $G$ and $H$.
Note that Proposition \ref{prop:2.2} has already been proved in this isoclinism sense.

Now we are ready to go back to group theory.

\section{Multiply induced characters: the proof of Theorem \ref{thm:1.1}} \label{sec:3}
Throughout Sections \ref{sec:3} and \ref{sec:4}, all characters are assumed to be complex irreducible characters.

In this section, we shall prove our main result Theorem \ref{thm:1.1}.
As is expected, we use Clifford theory. For that purpose, we recall that, for $N \triangleleft G $ and $\psi \in \irr (N)$,
we denote by $I_G (\psi) $ the inertia group of $\psi $ in $G$, and by $e(G/N)$ the integer $ \langle \psi , \chi_N \rangle_N $ 
for $\chi \in \irr (G \,|\, \psi)$ (see \cite[(6.2) Theorem]{MR0460423}).

We start with an easy, but crucial lemma.

\begin{lemma} \label{lem:3.1.5} 
    Let $H$ and $N$ be normal subgroups of $G$ such that $H \le N$.
    Let $\psi \in \irr (N)$, $\chi \in \irr ( G \,|\, \psi) $,
    and $\alpha \in \irr (H \,|\, \psi) $.
    By Clifford theory, there exist integers $e(G/N), e(N/H)$, and 
    $e(G/H)$ called inertia indices such that 
  \begin{equation} \label{eq:3.1}
    \chi_N = e(G/N) \sum_{j=1}^m \psi^{g_j}, \ \psi_H = e(N/H) \sum_{j=1}^{\ell} \alpha^{h_j},
 \ \chi_H = e (G/H) \sum_{j=1}^n \alpha^{k_j}.
  \end{equation}
Among these integers, there is a relation:
\begin{equation} \label{eq:3.2}
e (G/H) [G: I_G (\alpha )] = e (G/N) e(N/H) [G: I_G (\psi)] [N : I_N (\alpha )].
\end{equation}
\end{lemma}
\begin{proof}
Based on \eqref{eq:3.1}, 
we consider the following coset decompositions:
    \[
    G = \bigsqcup_{j=1}^m I_G (\psi ) g_j, \quad N = \bigsqcup_{j=1}^{\ell} I_{N} (\alpha ) h_j,
    \ G= \bigsqcup_{j=1}^n I_G (\alpha ) k_j 
    \]
    with $m= [G : I_G (\psi )], \  \ell = [N : I_N (\alpha )] $, and 
    $n = [G: I_G (\alpha )]$.
It follows also from \eqref{eq:3.1} that
\begin{equation} \label{eq:3.25}
    \chi (1) = e (G/N) m \psi (1), \ \psi (1) = e (N/H) \ell \alpha (1), \ \chi (1) 
= e(G/H) n \alpha (1).
\end{equation}
Since the degrees $\chi (1), \psi (1)$, and $\alpha (1)$ are all non-zero, \eqref{eq:3.2} follows 
immediately.
\end{proof}

We now recall the setting and assumptions of Theorem \ref{thm:1.1}.
We assume that a finite group $G$ satisfies the following properties:
\begin{itemize}
    \item $G$ is a solvable group;
    \item  There exist distinct non-trivial maximal normal subgroups $N_1$
     and $N_2$ of $G$ (hence $G=N_1 N_2$);
    \item  $G$ has a Camina triple $(G,N_1, M_1)$ including $N_1$.
\end{itemize}

Moreover we set $H = N_1 \cap N_2$. 
Under these assumptions, the following lemma holds.

\begin{lemma} \label{lem:3.1}
  If $\psi \in \irr (N_1  \,|\, M_1)$, then $\psi^G \in \irr (G)$.
\end{lemma}
\begin{proof}
Recall that $(G,N_1,M_1)$ is a Camina triple. If $\irr (N_1 \,|\,  M_1) = \emptyset$, then
the kernel of every irreducible character of $N_1$ contains $M_1$ and it follows that 
$\displaystyle M_1 \subset \bigcap_{\psi \in \irr (N_1)} \ker \psi =1$ by \cite[Lemma 2.21]{MR0460423}. 
This contradicts the definition of Camina triple (see Definition \ref{def:1.1}) 
and therefore, $\irr (N_1 \,|\,  M_1) \neq \emptyset$.

Let $\psi \in \irr (N_1  \,|\, M_1)$ and $\chi \in \irr (G \,|\, \psi )$. 
It follows from \cite[Lemma 5.1]{MR4166586} that 
$\chi (x) =0 $ holds for all $x \in G \setminus N_1$. This implies 
\[ \psi^G = [I_G (\psi ) : N_1]^{1/2} \chi  \] 
 as shown in \cite[Lemma 5.2]{MR4166586} (see also \cite[Theorem 3.1]{MR3150725}).
On the other hand, since $G$ is solvable and $N_1$ is maximal, the index $[G:N_1]$ is prime. Hence we conclude that a divisor $[I_G (\psi ) : N_1]$ of $[G:N_1]$ is equal to $1$.
 Hence $\psi^G$ is irreducible and equals $\chi $.
\end{proof}

Now we add the assumption on $I_G (\alpha ) $ in Theorem \ref{thm:1.1} and prove the theorem.

\begin{proof}[Proof of Theorem \ref{thm:1.1}]
Applying Lemma \ref{lem:3.1.5} to the case where $G \ge N_1 \ge H$ 
and $\psi \in \irr (N_1 \,|\, M_1)$, 
we obtain
\begin{equation} \label{eq:3.3}
e (G/H) [G: I_G (\alpha )] = e (G/N_1) e(N_1/H) [G: I_G (\psi)] [N_1 : I_{N_1} (\alpha )].
\end{equation}
Similarly for $G \ge N_2 \ge H$ with $\gamma \in \irr (N_2 \,|\, \alpha )$, we have
\begin{equation} \label{eq:3.4}
e (G/H) [G: I_G (\alpha )] = e (G/N_2) e(N_2/H) [G: I_G (\gamma)] [N_2 : I_{N_2} (\alpha )].
\end{equation}
By Lemma \ref{lem:3.1}, we have $\psi^G \in \irr (G)$ and thus, $I_G (\psi ) =N_1$ and 
$\chi (1) = [G:N_1] \psi (1)$ hold.
These equations imply $e (G/N_1)=1$ since $\chi (1)= e(G/N_1) [G:I_G (\psi)]\psi(1)$ by \eqref{eq:3.25}.
Hence \eqref{eq:3.3} becomes
\begin{equation} \label{eq:3.5}
e (G/H) [G: I_G (\alpha )] =  e(N_1/H) [G: I_{N_1} (\alpha )].
\end{equation}
By combining with \eqref{eq:3.4}, we get 
\[
e (G/H) [G: I_G (\alpha )] =  e(N_1/H) [G: I_{N_1} (\alpha )]=e (G/N_2) e(N_2/H) [G: I_G (\gamma)] [N_2 : I_{N_2} (\alpha )].
\]
Since $N_1$ is a maximal normal subgroup, we have $G' \le N_1$ and therefore $G/N_1$ is a cyclic group 
of prime order. This is also true for $G/N_2$. Thus, $N_1/H $ and $N_2/H$ are also cyclic groups of
prime order.
By \cite[Corollary 6.19]{MR0460423}, we have 
\begin{equation} \label{eq:3.5.5}
    e(G/N_1)=e(G/N_2)= e(N_1/H)=e(N_2/H)=1.
\end{equation}
Hence we obtain 
\begin{equation} \label{eq:3.6}
e (G/H) [G: I_G (\alpha )] =  [G: I_{N_1} (\alpha )]= [G: I_G (\gamma)] [N_2 : I_{N_2} (\alpha )].    
\end{equation}

We first consider the case where $|G/N_1|$ and $|G/N_2|$ are relatively prime and we assume 
$I_G (\alpha ) = H$. This implies that the induction $\alpha^G$ lies in $\irr (G)$.
We also have $I_{N_1} (\alpha ) =I_G (\alpha ) \cap N_1 = H$ and similarly $I_{N_2} (\alpha ) =H$.
It follows from the first equality of \eqref{eq:3.6} that $e(G/H)=1$.
The second equality of \eqref{eq:3.6} yields $ [N_2: H] [G: I_G (\gamma )] = [G: H]$.
Since $[G :N_2]$ is a prime and $[G:I_G (\gamma) ] $ divides $[G: N_2]$, either $I_G (\gamma ) =G$
or $I_G (\gamma ) =N_2$ holds. We shall show that the former is impossible.
Suppose to the contrary that $I_G (\gamma ) =G$ holds. Then $\gamma^G$ is not an irreducible character. 
On the other hand, $I_{N_2} (\alpha ) =H$ yields $\alpha^{N_2} \in \irr (N_2)$ and hence 
$\alpha^{N_2} = \gamma $. Therefore we obtain $\alpha^G =(\alpha^{N_2})^G =\gamma^G$. 
This is a contradiction. We thus conclude that $I_G (\gamma ) = N_2$. It follows from this that
$\gamma^G \in \irr (G)$ and $\gamma^G = \alpha^G = \psi^G$.

Next we consider the case where $\gcd (|G/N_1|, |G/N_2|) = p$ with $p$ prime.
The quotient group $G/H$ is an elementary abelian group of order $p^2$.

If $I_G (\alpha) =H$, then we have $I_{N_1} (\alpha ) = I_{N_2} (\alpha ) = H$ as above
and the equation \eqref{eq:3.6} yields
\[ e (G/H) [G: H ]=  [G: H]= [G: I_G (\gamma)] [N_2 : H].
\]
Hence we have $I_G (\gamma) = N_2$ since $I_G (\gamma) \ge N_2$.
Therefore we conclude $\gamma^G \in \irr (G \,|\, \alpha ) $ as expected.

If $I_G (\alpha) =G$, then a similar argument shows that 
\[
e (G/H) =  [G: N_1]= [G: I_G (\gamma)]. 
\]
This implies that $I_G (\gamma) = N_2$, and therefore $\gamma^G \in \irr (G \,|\, \alpha ) $.

Now we assume $H < I_G (\alpha) < G $. 
Let $N=I_G (\alpha )$. Our assumption implies $N \neq N_1, N_2$, and thus
$I_{N_1} (\alpha ) = I_{N_2} (\alpha ) =H$. Hence the equation \eqref{eq:3.6} yields
\[
e (G/H) [G: N] =  [G: H ]= [G: I_G (\gamma)] [N_2 : H].
\]
Since $[G : I_G (\gamma )] \le [G: N_2]$, we conclude $I_G (\gamma ) =N_2$. 
This yields $\gamma^G \in \irr (G) $. 

We next prove the converse. Thus we assume that $\psi \in \irr (N_1 | M_1)$ is given and that there exists $\gamma \in \irr (N_2)$ such that
$\psi^G= \gamma^G$. Let $\alpha \in \irr (H \,|\, \psi ) \cap \irr (H \,|\, \gamma )$.
As in the first half of the proof, the identities \eqref{eq:3.3} and \eqref{eq:3.4} hold. 
Our assumption yields $I_G (\psi) = N_1$ and $I_G (\gamma )=N_2$.  The equalities \eqref{eq:3.5.5} hold for the same reason.
Therefore, we obtain 
\begin{equation} \label{eq:3.9}
e(G/H) [G: I_G (\alpha )]  = [G : N_1][N_1: I_{N_{1}} (\alpha )] = [G : N_2] [N_2: I_{N_2} (\alpha )]
\end{equation}
instead of \eqref{eq:3.6}.

Suppose first that $\gcd ( |G/N_1| , |G/N_2|)=1$. 
We rewrite \eqref{eq:3.9} as
\[
e(G/H) [G: I_G (\alpha )]  = [G : I_G (\alpha ) ][ I_G (\alpha ): I_{N_{1}} (\alpha )] = [G : I_G (\alpha)] [I_G (\alpha ): I_{N_2} (\alpha )]
\]
and obtain 
\begin{equation} \label{eq:3.10}
e(G/H)   = [ I_G (\alpha ): I_{N_{1}} (\alpha )] =  [I_G (\alpha ): I_{N_2} (\alpha )].
\end{equation}
For $i=1,2$, we have an equality $ [ I_G (\alpha ): I_{N_{i}} (\alpha )] = [I_G (\alpha ) N_i : N_i] $ and this index is a divisor 
of $[G:N_i]$. Hence our assumption implies $ \gcd ([ I_G (\alpha ): I_{N_{1}} (\alpha )] , [ I_G (\alpha ): I_{N_{2}} (\alpha )] )=1$.
This yields $e(G/H)=1$ and $I_G (\alpha) = I_{N_1} (\alpha ) = I_{N_2} (\alpha)$.
Since $I_{N_1} (\alpha ) \cap I_{N_2} (\alpha) = I_{H} (\alpha)$, we have $I_G (\alpha) = I_H (\alpha ) =H$. This is the condition \ref{1.1.1}.

Next we consider the case where $\gcd ( |G/N_1| , |G/N_2|)$ is a prime number $p$.
We have to prove that $I_G (\alpha ) \neq N_i \ (i=1,2)$.

A similar argument as the above case shows \eqref{eq:3.10}.
Since $[I_G (\alpha): I_{N_1} (\alpha) ]$ is a divisor of $p$, we have either $e(G/H)=1$ or $e(G/H)=p$.

If $e(G/H)=1$, we have $I_G (\alpha ) = I_{N_1} (\alpha) = I_{N_2} (\alpha ) =H$ similarly as above.

Next assume that $e(G/H)=p$.
If $I_{N_1} (\alpha)= I_{N_2} (\alpha) $ holds, then it follows that $I_{N_1} (\alpha ) = 
I_{N_1} (\alpha) \cap  I_{N_2} (\alpha) = H $ and thus, $I_{N_2} (\alpha ) = H$ also holds.
If $I_G (\alpha ) = N_1$, then $N_1= I_G (\alpha ) \cap N_1 = I_{N_1} (\alpha ) = H$. Hence
$N_1 \cap N_2 = H = N_1$ yields $N_2 \ge N_1$. By the maximality of $N_1$ and $N_2$, we obtain $N_1 = N_2$
and this is a contradiction. If $I_G (\alpha ) =N_2$, a similar argument leads to a contradiction.

Assume now that $I_{N_1} (\alpha) \ne  I_{N_2} (\alpha)$.
Suppose to the contrary that $I_{G} (\alpha ) =N_1$. Then we have $I_{N_1} (\alpha ) = N_1$ and $I_{N_2} (\alpha ) = H$ and hence,
$I_{N_1} (\alpha )> I_{N_2} (\alpha ) $. This contradicts \eqref{eq:3.10}.
The case where $I_{G} (\alpha ) =N_2$ can be proved similarly.

This completes the proof of Theorem \ref{thm:1.1}.
\end{proof}

\section{Multiply induced characters and isoclinisms} \label{sec:4}
In this section, we prove Theorem \ref{thm:1.2}.
First we recall the definition of isoclinism from \cite{MR0003389}.

\begin{definition} \label{def:4.1}
  Let $G$ and $H$ be finite groups. 
  The groups $G$ and $H$ are \textit{isoclinic} if there exist isomorphisms
  $\eta : G /Z(G) \xrightarrow{\sim} H /Z (H) $
  and $ \xi : {G}' \xrightarrow{\sim} {H}'$ such
  that the following diagram is commutative:
\[
   \xymatrix{
  G/ Z(G) \times G/ Z(G) \ar[d]_{\eta \times \eta}
  \ar[r]^-{k_{G}} & {G}' \ar[d]^{\xi} \\
    H/ Z(H) \times H/ Z(H) \ar[r]^-{k_{H}} & {H}' ,
   }
\]
 where $k_{G} $ and $k_{H}$ are the commutator maps.
 If $G$ and $H$ are isoclinic, then we write $G \sim H$
 and we call the pair $(\eta, \xi)$ an
  \textit{isoclinism}.
\end{definition}

Isoclinism is an equivalence relation on finite groups that is coarser than isomorphism.
In every isoclinism class of finite groups, there is a representative called a \emph{stem group}.
A stem group satisfies, by definition, $Z(G) \le  G'$ and is of minimal order in the class.

We now assume that two finite groups $G$ and $H$ are isoclinic 
by the isoclinism $(\eta, \xi)$.

In \cite{MR0148734}, Weichsel constructed a group $C$ such that $C \sim G, \, C \sim H$,
and that both $G$ and $H$ are quotients of $C$.
The group $C$ is defined as a fiber product:
\begin{equation}  \label{eq:4.1}
C  = G \Yup H  =\{ (g,h) \in  G \times H \mid \eta (gG) = hH  \}.
\end{equation}
If we set
\[ Z_G = \{ (z_1, 1) \in C \mid z_1 \in Z (G) \} \text{ and }
Z_H = \{ (1, z_2) \in C \mid z_2 \in Z (H) \},
\]
then we can show that $C/Z_{H} \cong G$ and $C/Z_{G} \cong  H$.
From the fact $C' \cap Z_{H} = C' \cap Z_{G}=1$, it follows that 
the canonical surjections $C \rightarrow G$ and $C \rightarrow H$
induce isoclinisms $C \sim G$ and $C \sim H$.

The group $C$ mediates to describe $\irr (H)$ in terms of $\irr (G)$ under the assumption that
$G$ is a stem group.

\begin{proposition}[\protect{\cite[III.(5.6) and (5.7)]{MR681287}}] \label{prop:4.2}
Let the notation and the assumption being as above. 
In particular, let $G$ be a stem group.
For each irreducible representation $\rho $ of $G$, 
there exists $\mu^* \in \irr (C)_1$ and 
$ \tau_i \in \irr (H)_1$ such that all irreducible representations
$\tilde{\rho} $ of $H$ are of the form
\begin{equation} \label{eq:4.2}
\tilde{\rho} (h) = \mu^* (g , h) \tau_i (h) \rho (g). 
 %= \hat{\rho} (g,h) \tau_i (h).
\end{equation}
for any $g \in G$ and $h \in H$ satisfying $\eta (g Z(G)) = h Z(H)$.

Moreover, for $\rho_1, \rho_2 \in \irr (G)$, the following statements are equivalent: 
\begin{enumerate}
\item $\tilde{\rho_1} \tau_i =\tilde{\rho_2} \tau_j$ holds;
\item $\rho_1= \rho_2$ and $\tau_i \equiv \tau_j \pmod{\irr (H/Z(H)H')}$ hold.
\end{enumerate}
\end{proposition}

It is shown that we can take $\tau_i \ (i=1,\ldots, b)$ as a transversal of $\irr (H/Z(H)H')_1$ in $\irr (H)_1$
with $b=[Z (H): Z (G)]$.

Within this construction, we may consider \eqref{eq:4.2} as an equality of characters of $C$. 
It is easy to see $\eqref{eq:4.2}$ does not depend on the choice of $g$ and $h$.

The following proposition contains a key observation for the proof of Theorem \ref{thm:1.2}.

\begin{proposition} \label{prop:4.3}
Assume that $G$ and $H$ are isoclinic by isoclinism $(\eta, \xi)$
and that $G$ is a stem group.
Let $\rho$ be an irreducible representation of $G$ and $\tilde{\rho }$ a corresponding 
irreducible representation of $H$ as in \eqref{eq:4.2} and $\chi$ 
and $\tilde{\chi}$ the corresponding characters.

If two groups $G_1$ and $H_1$ satisfy $Z(G) \le G_1 \le G$ and $Z(H) \le H_1 \le H$
and $G_1/ Z(G) \cong H_1/ Z(H) $ by $\eta$, then the following statements are equivalent: 
\begin{enumerate}
    \item $\chi$ is induced from an irreducible character of $G_1$;
    \item $\tilde{\chi} $ is induced from an irreducible  character of $H_1$.
\end{enumerate}
\end{proposition} 
\begin{proof}
The equation \eqref{eq:4.2} yields that
\[
\tilde{\chi} (h)= \mu^* (g , h) \tau_i (h) \chi (g)  \]
for $(g,h) \in C$. 
We may regard all characters appearing in this equality as characters of $C$.

Now suppose that $\chi = \xi^G$ with $\xi \in \irr (G_1)$, then we obtain
\[
\tilde{\chi} = ((\mu^*)_{C_1}(\tau_i)_{C_1}\xi)^C,
\]
where $C_1 = \{(g,h) \in C \mid g \in G_1\}$, by \cite[(38.5) Theorem]{MR0144979}.
We note that, for $(g,h) \in C_1$,
\[
\tilde{\chi}(g,h) = \tilde{\chi}(g).
\]
On the other hand, the character $(\mu^*)_{C_1}(\tau_i)_{C_1}\xi$ vanishes on
$Z_G = \{(x,1) \in C \mid x \in Z(G)\} \le Z(C)$, and therefore,
is induced from $C_1/Z_G \cong H_1$.
This shows that $\tilde{\chi}$ is induced from a character in $\irr(H_1)$.

The reverse implication is proved similarly and we omit the proof.
\end{proof}

We note that if a character $\chi \in \irr (G)$ is induced from an irreducible 
character of a subgroup $H$, then $H \supset Z(G)$ always holds, since the induced character vanishes
outside $H$.

The following corollary, which will be used in Section \ref{sec:5.3},
follows immediately from the above proposition.

\begin{corollary} \label{cor:4.4}
    Under the assumptions and notation in Proposition \ref{prop:4.3}, 
    $\chi $ is a primitive character if and only if so is $\tilde{\chi}$.
\end{corollary}

To prove Theorem \ref{thm:1.2}, we have to remove the condition that $G$ is a stem group.
Indeed, we can prove the following stronger theorem than Theorem \ref{thm:1.2}.

\begin{theorem}
    Let $G$ and $H$ be isoclinic groups. The group $G$ has a multiply induced character if and only
    if so does $H$.
\end{theorem}
 \begin{proof}
   We keep the notation in Proposition \ref{prop:4.3}.

   Let $G_0$ be a stem group in the isoclinism class of $G$ (hence of $H$).
   We apply Proposition \ref{prop:4.3} to both isoclinisms $G \sim G_0$ by $(\xi, \eta )$ 
   and $H \sim G_0$ by $(\xi',\eta' )$.
   If $\tilde{\chi }\in \irr (G) $ is induced both from characters of subgroups $U_1$ and $U_2$ 
   containing $Z(G)$, then, by Proposition \ref{prop:4.3}, 
   the corresponding character $\tilde{\chi} \in \irr (G_0)$  
   is induced from those of $\eta (U_1)$ and $\eta (U_2)$.
   If $\hat{\chi} \in \irr (H)$ is a character corresponding to $\chi$ by \eqref{eq:4.2}, then
   it is induced simultaneously from $\eta'^{-1} \eta (U_1)$ and  $\eta'^{-1} \eta (U_2)$.
   
   The converse assertion is also valid, since the isoclinism is an equivalence relation.
 \end{proof}

At the end of this section, we show that there is a one-to-one correspondence of Camina triples 
in isoclinic groups. This result, in turn, gives a variant of Theorem \ref{thm:1.2}.

\begin{theorem}  \label{prop:4.5}
Let $G$ and $H$ are isoclinic groups by the isoclinism $(\eta, \xi )$.
Suppose that $(G,N_1,M_1)$ is a Camina triple of $G$. If we define subgroups $N_2$ and $M_2$ of $H$ by
$\eta (N_1/Z (G)) = N_2 /Z (H)$ and $ \xi (M_1) = M_2$, then $(H,N_2, M_2)$ is a Camina triple 
of $H$.
\end{theorem}

We need the following lemma to justify some statements in the theorem.

\begin{lemma} \label{lem:4.5}
If $(G,N,M)$ is a Camina triple, then $Z (G) \le N$ and $M \le G'$ hold.
\end{lemma}
\begin{proof}
   The first assertion is shown in \cite[Lemma 2.4 (i)]{MR3150725}. 
   By Theorem 2.1 of the same paper, for all $g \in G \setminus N$ and $z \in M$, there exists $y \in G$
   such that $[g,y]=z$. This readily implies the second assertion.
\end{proof}

\begin{proof}[Proof of Theorem \ref{prop:4.5}]
    By Lemma \ref{lem:4.5}, $N_2$ is a normal subgroup of $H$ containing $Z (H)$. Also 
    $M_2$ is a normal subgroup of $H$ contained in $H'$ since,
     if $\eta (gZ (G)) = h Z(H) \ (g \in G, h \in H)$, then it is easy to see that
    $\xi (m^g) = \xi (m)^h $ holds for $m \in G'$.

    It remains to prove that every $h \in H-N_2$ is conjugate to all of $hM_2$.
    For any $m_2 \in M_2$, we take $g \in G$ and $m_1 \in M_1$ satisfying
    $\eta(gZ(G))=hZ(H)$ and $\xi(m_1)=m_2$.
    Since $h \not\in N_2$, we see that $g \not\in N_1$.
    As $(G,N_1,M_1)$ is a Camina triple, there exists $x \in G$ such that
    $gm_1=g^x$, namely $m_1=[g,x]$.
    Sending both sides by $\xi$ and using the definition of isoclinism,
    we obtain $m_2=[h,y]$ for an element $y \in H$ satisfying
    $\eta(xZ(G))=yZ(H)$. This means that $h$ is conjugate to $hm_2$.
\end{proof}

% If a Camina triple of $G$ satisfies the condition in Theorem \ref{thm:1.1}, then
% so does every isoclinic group.

\section{Number fields sharing an Artin $L$-function} \label{sec:5}
In this section, we give several explicit examples of Theorem \ref{thm:1.1}.
The following examples illustrate what is in the scope of Theorem 1.2 and what is out of the scope.
They show that its hypotheses are close to being optimal.

\subsection{Frobenius groups} \label{sec:5.1}
Let $G=F_{q} \cong \mathbb{F}_q \rtimes  \mathbb{F}_q^{\times}$ be a Frobenius group, where 
$\mathbb{F}_q$ is the finite field of $q$-elements with a prime number $q$.
 We assume that there exist two distinct prime divisors $p$ and $\ell$
of $q-1$. Then there are maximal normal subgroups $N_1$ and $N_2$ of $G$ of index $p$ and $\ell$,
respectively.
To be more precise, let us fix a presentation of $G$:
\[
G=  \langle a,b \mid a^q = b^{q-1} =1,
bab^{-1} = a^g    \rangle,
\]
where $g$ is a primitive root modulo $q$.
Two maximal normal subgroups are, then, given by
\[
N_1=\langle a,b^p  \rangle , \quad N_2=\langle a,b^{\ell}  \rangle.
\]
It is easy to see that both 
$(G,N_1, \langle a \rangle )$ and $(G,N_2, \langle a \rangle )$ are Camina triples.
We also have $H=N_1 \cap N_2 = \langle a,b^{p\ell }  \rangle$.
Both maximal normal subgroups are metacyclic and, therefore, the character tables can be computed 
as described in \cite[Section 47]{MR0144979}.
Indeed, there are $p$ characters $\psi $ in $\irr (N_1)$ of degree $(q-1)/p$ and 
$\ell $ characters $\gamma$ in $\irr (N_2)$ of degree $(q-1)/\ell $. Both of them are induced from a 
linear character of $\langle a \rangle $. Therefore every irreducible constituent $\alpha$ of 
the restriction of these characters to $H$ is also induced from $\langle a \rangle$ and we have
$I_G (\alpha ) = H$. Hence we can apply Theorem \ref{thm:1.1} and obtain $\psi^G = \alpha^G=\gamma^G$ for 
every pair of 
$\psi \in \irr (N_1 \,|\, \alpha) $ and $\gamma \in \irr (N_2 \,|\, \alpha)$.

If $K/\Q$ is a $G$-extension, then the equality 
\[
L (s, \psi, K/K^{N_1}) = L (s, \gamma , K/K^{N_2})
\]
holds.
Note that the absolute degrees of $K^{N_1}$ and $K^{N_2}$ are distinct and until now, 
no such example has been known.
Note also 
that if $q-1$ has more prime divisors, there are more $L$-functions involved in the above equality.

\begin{example}
    Let us consider the Frobenius group $G=F_7$ of order $42$.
    Hence we let $q=7, p=2, \ell =3 $ in the above notation.
    Let $K$ be the splitting field of the polynomial $f(X) =X^7 -2 \in \Q [X] $.
    We can show that the Galois group of $K/\Q$ is isomorphic to $G$.
    The subfields in which we are interested are
    \begin{align*}
        K^{N_1} & = \Q (\sqrt{-7}),\\
         K^{N_2} &:  X^3 - X^2 - 2X + 1 , \\
         K^{H} & = \Q (\zeta_7 ).
    \end{align*}
    Of course, $K^{N_2}$ is a cubic cyclic subfield of $\Q (\zeta_7)$.
        Six nontrivial linear characters $\alpha_i \, (i=1,\ldots , 6)$ of $H$ are conjugate under $G$,
    and therefore, there are three degree-$2$ characters $\psi_j \, (j=1,\ldots ,3)$ in $\irr (N_1)$ 
    and two degree-$3$ characters $\gamma_k \, (k=1,2)$ in $\irr (N_2)$ such that 
    \[
    L (s, \alpha_i, K/K^H) = L (s,\psi_j , K/K^{N_1}) = L (s, \gamma_k, K/K^{N_2})
    = L (s, \chi, K/\Q)
    \]
    holds with $\chi = \alpha_i^G $. The Dirichlet series of the above $L$-functions is
    \[
    L (s, \chi, K/\Q ) = \frac{1}{1^s} - \frac{1}{29^s} -\frac{1}{43^s} -\frac{1}{71^s}+ \cdots .
    \]
\end{example}

As we have seen, this is a typical situation to which we can apply Theorem \ref{thm:1.1} with the condition (i).

\subsection{Extraspecial groups} \label{sec:5.2}
Let $p$ be a prime number.
Extraspecial $p$-groups are nonabelian $p$-groups such that $\Phi (G) = G' = Z(G)$ 
is a cyclic group of prime order $p$, where $\Phi (G)$ is the Frattini subgroup of $G$.
If $G$ is an extraspecial group of order $p^{2m+1}$, then
the quotient group $G/Z(G)$ is an elementary abelian group of order $p^{2m}$ and 
the maximal abelian normal subgroups are of order $p^{m+1}$ (see \cite[III.13.7 Satz]{MR0224703}).
Every subgroup of $G$ of index $p$ is always a normal subgroup of $G$ and every such group $N$
sits in a Camina triple $(G,N,Z(G))$.
Let $N_1$ be such a group. It contains a maximal abelian normal subgroup $A_1$.
Then by \cite[III.13.7 Satz f)]{MR0224703}, there exists a maximal abelian 
normal subgroup $A_2$ so that  $G=A_1 A_2 $ and $A_1 \cap A_2 = Z(G)$ hold.
Let $N_2$ be a maximal normal subgroup of $G$ containing $A_2$.
From $A_1 A_2 = G$, it follows $N_1 \neq N_2$.     
The group $G$ has $p-1$ irreducible faithful 
characters of degree $p^m$, which are induced from a character 
of a maximal abelian normal subgroup that is an extension of a non-trivial linear 
character of $Z(G)$ (see \cite[V.16.14 Satz]{MR0224703}).
Let $\psi \in \irr (N_1 \,|\, Z(G))$. There exists $\alpha \in \irr (A_1)$ such that
$\alpha^{N_1}=\psi$ and $\alpha_{Z(G)} \ne 1$ hold.
Put $H=N_1\cap N_2$ and $\lambda=\alpha_{Z(G)}$.
Let $\beta \in \irr (H \,|\, \psi)$. Since $G$ has a unique irreducible character
lying above $\lambda$ (see \cite[V.16.14 Satz]{MR0224703}),
we have $\beta^G=e\psi^G$ for some positive integer $e$.
By Frobenius reciprocity and the proof of \cite[(6.2) Theorem]{MR0460423}, we obtain
\[
[I_G(\beta):H]=\langle \beta^G,\beta^G\rangle_G=e^2.
\]
Since $[G:H]=p^2$, this implies $I_G(\beta)=H$ or $G$.
Therefore $\beta$ satisfies condition (ii) of Theorem \ref{thm:1.1}.
Hence there exists $\gamma \in \irr (N_2)$ such that $\gamma^G=\psi^G$.
This recovers a part of Proposition \ref{prop:2.2}.

\begin{remark}
As in this case, we observe that 
the more Camina triples there exist in $G$, the more number fields share an $L$-function.
On the other hand, the existence of Camina triples is not a necessary condition for the existence of multiply induced characters.
See Example \ref{ex:5.7} below.
\end{remark}

For explicit examples of number fields sharing an $L$-function in this extraspecial case, 
see \cite[Example 3.4]{MR4474773}.

\subsection{Primitive characters} \label{sec:5.3}
Examples constructed so far are monomial groups. Thus any multiply induced characters are inductions
from linear characters of its subgroups and therefore, the $L$-function shared by number fields
are Hecke $L$-functions with certain ray class characters.
It is natural to ask whether there exist a multiply induced character induced from 
primitive characters of subgroups.
In this case, number fields share a nonabelian Artin $L$-function. Indeed, this question also motivates 
the research in this paper.

As in Theorem \ref{thm:1.1}, suppose that $\chi \in \irr (G)$ is multiply induced 
from primitive characters  $\psi \in \irr (N_1)$ and $\gamma \in \irr (N_2)$ 
with maximal normal subgroups $N_1$ and $N_2$.
It is obvious that both $N_1$ and $N_2$ are non-monomial, and hence so is $G$
because $\psi^G = \chi $ is not a monomial character.
By computation, we found that there is no non-monomial group of order less than $96$ 
to which we can apply Theorem \ref{thm:1.1}.
Among the groups of order $96$,
there are five groups up to isomorphism having multiply induced 
    primitive characters. The five groups are divided in two isoclinism classes.

\begin{example} \label{ex:5.2}
     The first isoclinism class consists of the groups with GAP id. $(96,190),(96,191),
    (96,193)$. The first group $G_1$ is a Schur covering group of $\mathrm{GL}(2,3)$.
    % Schur covers are isoclinic. But isoclinic groups of same order may not be a Schur cover.
    The group $G_1$ has a unique normal subgroup $N_1$ isomorphic to $C_2 \times \mathrm{SL}(2,3)$.
    There exists only one Camina triple $(G,N_1, Z(G))$ in $G$ and 
    there are six degree $2$ characters $\psi_i \, (i=1,\ldots , 6)$ in $\irr (N_1 \,|\, Z(G))$.
    Let $N_2$ be a unique maximal normal group of $G_1$ isomorphic to $\mathrm{GL}(2,3)$.
    In this case, $H=N_1 \cap N_2$ coincides with a normal subgroup isomorphic to $\mathrm{SL}(2,3)$.
    Let $\psi_1$ and $\psi_2$ be the rational characters in $\irr (N_1 \,|\, Z(G))$.
    They share  $\alpha \in \irr (H \,|\, \psi_1) \cap \irr (H \,|\, \psi_2 )$ satisfying
    the assumptions of Theorem \ref{thm:1.1}: $I_G (\alpha) = G$. Therefore, the theorem 
    implies that there exists 
    $\gamma \in \irr (N_2) $ such that $\gamma^G = \psi_1^G$ and $\gamma_H = \alpha$.
    Note here that $(\gamma^G)_{N_1} = \psi_1 + \psi_2$ and hence $\gamma^G = \psi_2^G$
    also holds.
    There is another maximal normal subgroup $N_3$ isomorphic to $\mathrm{CSU}(2,3)$,
    which is isoclinic to $N_2$. A similar argument applies to $N_3$ and there exists 
    $\delta \in \irr (N_3)$ such that $\delta^G = \psi_1^G$. Here $\delta $ is also 
    a primitive character (see Lemma \ref{lem:5.5} below).
\end{example}

\begin{example} \label{ex:5.3}
    The second isoclinism class consists of the groups $Q_8.A_4=(96,201)$ and $D_4.A_4=(96,202)$.
    They have four Camina triples. One of them is of the form $(G,\mathrm{SL}(2,3),Z(G))$ This does not 
    satisfy the assumption of Theorem \ref{thm:1.1}, since $[G:\mathrm{SL}(2,3)]=4$ is not square-free.
    The other three triples are of the form $(G,N, Z(G) ) $ with $[G:N ]=2$.
    The intersection of $N$'s coincides with $H \cong \mathrm{SL}(2,3)$.
    We can show that all six characters $\psi$ of degree $2$ of each 
    $N$ satisfy $\psi_H \in \irr (H)$,
    and $I_G ( \psi_H ) = G$ and hence the condition of Theorem \ref{thm:1.1} is satisfied.
    All six characters $\psi$ (they form three conjugacy classes) 
    induces three irreducible characters of $G$, each of which is induced multiply
    from three $N$'s. The situation is described by the following graph:
    
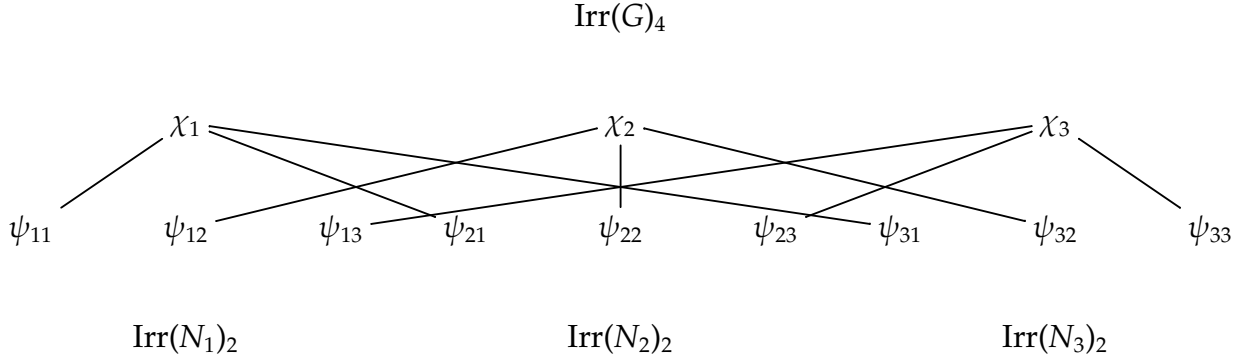
\[
    \xymatrix{
         &  &  & & \irr (G)_4 & &&& \\
       & \chi_1 &  &  & \chi_2 & & & \chi_3 &  \\
    \psi_{11} \ar@{-}[ru] & \psi_{12} \ar@{-}[rrru] & \psi_{13} \ar@{-}[rrrrru] & \psi_{21} \ar@{-}[llu] & \psi_{22} \ar@{-}[u]& \psi_{23} \ar@{-}[rru] 
    & \psi_{31} \ar@{-}[lllllu]& \psi_{32} \ar@{-}[lllu] & \psi_{33} \ar@{-}[lu] & \\
              & \irr (N_1)_2 &          &           & \irr (N_2)_2 &     &       & \irr (N_3)_2 &  &
    }
    \]

    Let $f(X)=x^{16} + 8x^{14} + 28 x^{12} + 60 x^{10} + 90 x^8 + 88x^6 + 52x^4 + 12x^2 + 1 \in \Q [X]$. 
    The Galois group of the splitting field $K$ of $f(X)$ over $\Q$ is isomorphic to $G=(96,202)$ and contains
    three quadratic subfields $k=\Q (\sqrt{2}), \Q (\sqrt{-2}), \Q (\sqrt{-1})$.
    Three $L$-functions attached to the irreducible characters of degree $4$ of $\gal (K/\Q)$ 
    are each shared by the Artin $L$-functions of Galois extensions of these quadratic subfields. One of the $L$-series shared by 
    three extensions is given by
    \[
    L(s,\chi_2, K/\Q ) = \frac{1}{1^s} +\frac{\omega}{9^s} -\frac{2\omega^2}{17^s} -\frac{\omega^2}{25^s}- \frac{1}{49^s} +\cdots 
    \]
    with $\omega = \exp \left( \frac{2\pi \sqrt{-1}}{6} \right)$.
    The coefficients tell us that this is not a Hecke $L$-function.
\end{example}

The following lemma deduced from Proposition \ref{prop:4.3} sometimes helps to check the primitivity of a character.
\begin{lemma} \label{lem:5.5}
    Suppose that a character $\chi \in \irr (G)$ is multiply induced from its normal 
    subgroups $N_1$ and $N_2$. If $N_1$ and $N_2$ are isoclinic and
    $\chi $ is induced from a primitive character of $N_1$, then $\chi $ is induced also
    from a primitive character of $N_2$.
\end{lemma}

It is natural to ask whether we can construct aimed Artin $L$-functions associated with 
primitive characters of arbitrary large degree.
The following lemma due to Hekster helps us to increase the degree of primitive characters.

\begin{lemma}[\protect{\cite[(4.5) Theorem]{MR783006}}] \label{lem:5.2}
    Let $G=G_1 \times G_2$ be a direct product of groups, 
    $\chi_i \in \irr (G_i) \, (i=1,2)$ and $\chi = \chi_1 \chi_2$.
    If $\chi_1$ and $\chi_2$ are primitive, then $\chi$ is primitive.
\end{lemma}

%\begin{proposition}
%Suppose that $G_1$ has an irreducible primitive character $\psi$ and $G_2$ has a multiply induced irreducible
%character $\chi=\phi_i^G \, (i=1,\ldots , m)$ with primitive character $\phi_i$ of subgroups $H_i \le G_2$.
%Then the characters $\psi  \phi_i \in \irr (G_1 \times H_i)$ are 
%irreducible primitive character and they induce $\psi  \chi \in \irr (G_1 \times G_2)$.
%\end{proposition}
%\begin{proof}
%The first assertion follows readily from Lemma \ref{lem:5.2}.
%The last assertion is also valid, since $(\psi  \phi_i)^G = \psi^G \phi_i^G = \psi  \chi$ for all $i$.
%\end{proof}

\begin{example} \label{6}
Let $G_1 = (96,190)$ as in Example \ref{ex:5.2} and $G_2 = D_4 = 2^{1+2}_{+}$ as considered
in Section \ref{sec:5.2}. 

The normal subgroup $N_1 \cong C_2 \times \mathrm{SL}(2,3)$ of $G_1$ has 
$6$ degree-$2$ characters $\psi_i \, (i=1,\ldots , 6)$ and they are all primitive, since they are of the form $\phi \chi $
with $\phi \in \irr (\mathrm{SL}(2,3))_2$ and $\chi \in \irr (C_2)$.

On the other hand, the unique $\gamma \in \irr (D_4)_2$ is multiply induced from $\delta_j 
\in \irr (H_j) \, (j=1,2,3)$ of the subgroups $H_j$ of order $4$.

Lemma \ref{lem:5.2} implies that the characters $\psi_i \delta_j \, (i=1,\ldots ,6, j=1, \ldots, 3)$ 
of $N_1 \times H_i $ are primitive and induce an irreducible character of $G_1 \times D_4$, since 
$(\psi_i \delta_j)^{G_1 \times D_4} =  \psi_i^{G_1} \delta_j^{D_4}$.
Hence $L (s,\psi_i \delta_j,  K/K^{N_1 \times H_j})$ of primitive characters are identical.

By taking direct powers of this example, 
we obtain $L$-functions with primitive characters of arbitrary large degree.
 \end{example}

\subsection{Induction from non-normal groups}
Although our theorems produce many examples of groups with multiply 
induced characters, our theorems only concern induction from normal subgroups.
In the end of this paper, we give an example of multiply induced characters 
involving non-normal subgroups.

\begin{example} \label{ex:5.7}
The smallest such groups are $\mathrm{GL}(2,3)=(48,29)$ and $\mathrm{CSU}(2,3)=(48,28)$.
They are isoclinic and have no Camina triple.
Here we let $G=\mathrm{GL}(2,3)=(48,29)$. Obviously $N=\mathrm{SL}(2,3)$ 
is a normal subgroup of $G$. The group $G$ has a non-normal subgroup $H$ isomorphic 
to $D_6$ of order $12$.
The degree-$4$ irreducible character $\chi $ of $G$ is induced from 
linear characters $\beta$ of $H$ whose restriction to $Z(G)$ are nontrivial, and 
a non-rational degree-$2$ character $\phi $ of $N$.

A number-theoretic conclusion is the following. Inside $G$-extension $K/\Q$, we have
an equality between a non-abelian Artin $L$-function of $N$-extension 
associated to $\phi$ and a Hecke $L$-function of $H$-extension associated 
to $\beta$ whose kernel is isomorphic to $S_3$, and hence essentially a 
quadratic extension $K^{\ker \beta}/K^H$:
\[
 L (s, \phi, K/K^N) = L (s,\beta, K^{\ker \beta }/K^H).
\]
For example, the polynomial $f(X)=x^8 + x^7 - 3x^6 + x^5 + 8x^4 + x^2 + 7x + 1$ has the Galois 
group isomorphic to $G$, which gives rise to an Artin representation of degree $4$ with 
Artin conductor $283^2$.
Using Magma, the above $L$-series is given as a Dirichlet series
\[
L (s, \chi , K/\Q ) = \frac{1}{1^s} + \frac{1}{7^s} - \frac{1}{11^s} -\frac{1}{13^s}  -\frac{1}{16^s}
+\frac{1}{23^s} + \cdots.
\]
The least prime splitting completely in $K$ is $643$.
\end{example}

\begin{remark}
Two groups in Example \ref{ex:5.7} have a common feature.
While they do not contain a Camina triple, 
they do contain a nontrivial, proper subgroup $H$ and a proper normal subgroup $L$ of $H$ for which $H^g \cap H \le L$
for all $g \in G \setminus H$. Such groups are first studied by Wielandt as a generalization of Frobenius group and thus, sometimes called Frobenius-Wielandt groups.
Although Burkett and Lewis \cite[Theorem D]{MR4166586} give some information on the representation of such groups,
the above groups are not in the scope of the result.
\end{remark}

%\bibliographystyle{/Users/m.kida/Dropbox/doc/paper/mybib}
%\bibliography{/Users/m.kida/Dropbox/doc/paper/copied-paper}

\providecommand{\bysame}{\leavevmode\hbox to3em{\hrulefill}\thinspace}

\begin{flushleft}
Masanari Kida  (\texttt{kida@rs.tus.ac.jp})\\
 Kyohei Matsukawa \\
 Hikari Watanabe \\
 Department of Mathematics \\
 Faculty of Science Division I \\
Tokyo University of Science \\
Kagurazaka 1-3, Shinjuku, Tokyo, Japan
\end{flushleft}

\end{document}